\documentclass[12pt]{amsart}
\usepackage{amssymb, amscd, amsmath, amsthm, epsf, epsfig, latexsym, enumerate}

\newtheorem{theorem}{Theorem}
\newtheorem{lemma}[theorem]{Lemma}
\newtheorem{cor}{Corollary}[theorem]

\newtheorem*{thm}{Theorem}

\begin{document}

\title{solvable normal subgroups of 2-knot groups}
\author{J.A.Hillman}
\address{School of Mathematics and Statistics\\
     University of Sydney, NSW 2006\\
      Australia }

\email{jonathan.hillman@sydney.edu.au}
 
\begin{abstract}
If  $X$ is an orientable,  strongly minimal $PD_4$-complex and
$\pi_1(X)$ has one end then it has no nontrivial locally-finite normal subgroup.
Hence if $\pi$ is a 2-knot group then (a)
if $\pi$ is virtually solvable then either $\pi$ has two ends or $\pi\cong\Phi$,
with presentation $\langle{a,t}|ta=a^2t\rangle$,
or $\pi$ is torsion-free and polycyclic of Hirsch length 4;
(b) either $\pi$ has two ends, or
$\pi$ has one end and the centre $\zeta\pi$ is torsion-free,
or $\pi$ has infinitely many ends and $\zeta\pi$ is finite; and (c) 
the Hirsch-Plotkin radical $\sqrt\pi$ is nilpotent.
\end{abstract}



\maketitle

A 2-knot is a topologically locally flat embedding of $S^2$ in $S^4$.
The main result of  this note is that if the group of a 2-knot is virtually solvable 
then it either has finite commutator subgroup or is torsion-free.  
All such groups are known, and in the torsion-free case 
the knot exteriors are determined up to homeomorphism by the knot group $\pi$
together with the $Aut(\pi)$-orbit of a meridian.
With the exception of one group, 
such knots may be constructed by elementary surgery on a section of
the mapping torus of a diffeomorphism of a 3-manifold; 
the cocore of the surgery is then a smooth 2-knot in a smooth homotopy 4-sphere.
(The exceptional group is the group $\Phi$ of Examples 10 and 11 in \cite{Fo},
which are ribbon knots.)
In many cases the homotopy 4-sphere is standard,
but the question remains open in general.
(This construction is perhaps the simplest source for possible counter-examples to the
smooth 4-dimensional Poincar\'e conjecture \cite{Go}.)

The first section establishes some notation.
In \S2 we show that if $X$ is an orientable $PD_4$-complex
such that $\pi_1(X)$ has one end and the equivariant intersection pairing 
on $\pi_2(X)$ is 0 then $\pi_1(X)$ has no nontrivial locally-finite normal
subgroup.
This is used in \S3 to show that if a 2-knot group $\pi$ is virtually solvable 
then either $\pi'$ is finite or $\pi\cong\Phi={Z}*_2$, 
with presentation $\langle{a,t}|ta=a^2t\rangle$,
or $\pi$ is torsion-free polycyclic and of Hirsch length 4.
(The final family of such groups was found in \cite{HH}.)
More generally, if $S$ is an infinite solvable normal subgroup 
and $\pi$ is not itself solvable
then $S\cong\mathbb{Z}^2$ 
or is virtually torsion-free abelian of rank 1.
We show also that the Hirsch-Plotkin radical $\sqrt\pi$ of every 2-knot group 
is nilpotent.
Finally we consider the centre $\zeta\pi$.
If $\pi$ has one end then  $\zeta\pi\cong\mathbb{Z}^2$ 
or is torsion-free, of rank $\leq1$.
In particular, this is so if the commutator subgroup $\pi'$ is infinite and $\zeta\pi$ has rank $>0$. 
If $\pi$ has two ends $\zeta\pi$ has rank 1, and may be either $\mathbb{Z}$ or $\mathbb{Z}\oplus{Z/2Z}$,
while if $\pi$ has infinitely many ends $\zeta\pi$ is finite.
We extend a construction of \cite{Yo80} to give examples
with $\sqrt\pi$ cyclic of order $q$ or $2q$, with $q$ odd.

The group $\Phi$ is the only nontrivial higher dimensional knot group 
which is solvable and has deficiency 1.
In \S4 we sketch briefly how the determination of 2-knots 
with this group might extend to other groups with geometric dimension 2, 
and perhaps to all with deficiency 1.
In view of the present limitations on 4-dimensional surgery techniques,
we can expect only a classification up to $s$-concordance 
(relative $s$-cobordism of knots).
With this proviso, we can show that a 2-knot whose group $\pi$ 
is a 1-knot group is determined up to $s$-concordance,
reflections and the Gluck ambiguity 
(whether a knot is determined by its exterior \cite{Gl})
by $\pi$ and the $Aut(\pi)$-orbit of a meridian (Theorem 5).
In particular, a 2-knot  with group $\pi$ and meridian $\mu$  
is $s$-concordant to the Artin spin of a fibred $1$-knot 
if and only if $\pi'$ is finitely generated, 
$C_\pi(\mu)\cong\mathbb{Z}^2$ and $(\pi,C_\pi(\mu))$ is a $PD_3$-pair
(Corollary 6.1).

\section{notation}

If $G$ is a group let $\zeta{G}$ and $ G'$ be the centre
and commutator subgroup of $G$, respectively.
Let $C_G(H)$ be the centralizer of a subgroup $H\leq{G}$.
The product of locally-nilpotent normal subgroups of $G$
is again a locally-nilpotent normal subgroup, 
by the Hirsch-Plotkin Theorem,
and the Hirsch-Plotkin radical $\sqrt{G}$ 
is the (unique) maximal such subgroup.
(See Proposition 12.1.2 of \cite{Ro}.)
In general, $\sqrt{G}$ is elementary amenable, but need not be nilpotent.

If $E$ is an elementary amenable group then it has a well-defined Hirsch length
$h(E)\in\mathbb{N}\cup\{\infty\}$. (See \cite{HL} or Chapter 1 of \cite{Hi}.)

If $X$ is a space let $\beta_i^{(2)}(X)$ be the $i$th $L^2$ Betti number of $X$,
and let $\beta_i^{(2)}(G)=\beta_i^{(2)}(K(G,1))$.
(See \cite{Lu} for a comprehensive exposition of  $L^2$-theory.)

\section{$PD_4$-complexes with $\chi=0$}

A $PD_4$-complex $X$ with fundamental group $\pi$ is {\it strongly minimal} 
if the equivariant intersection pairing on $\pi_2(X)$ is 0,
equivalently, if the homomorphism from $H^2(\pi;\mathbb{Z}[\pi])$
to ${H^2(X;\mathbb{Z}[\pi])}$ induced by the
classifying map $c_X:X\to{K(\pi,1)}$ is an isomorphism \cite{Hi09}.
Note that $X$ is aspherical if and only if it is strongly minimal and $H^i(\pi;\mathbb{Z}[\pi])=0$ for $i\leq2$.

\begin{lemma}
Let $X$ be a finite $PD_4$-complex with fundamental group $\pi$.
If $\chi(X)=0$ and $\beta_1^{(2)}(\pi)=0$ then $X$ is strongly minimal.
\end{lemma}

\begin{proof}
Since $X$ is a finite complex the $L^2$-Euler characteristic formula holds,
and so $\chi(X)=\beta_2^{(2)}(X)-2\beta_1^{(2)}(X)$.
Hence $\beta_2^{(2)}(X)=0$ also.
Since $\beta_1^{(2)}(X)=\beta_1^{(2)}(\pi)=0$,
$\beta_2^{(2)}(X)\geq\beta_2^{(2)}(\pi)\geq0$ and $\chi(X)=0$,
it follows that $\beta_2^{(2)}(X)=\beta_2^{(2)}(\pi)$.
Hence $H^2(c_X;\mathbb{Z}[\pi])$ is an isomorphism, 
by part (3) of Theorem 3.4 of \cite{Hi}.
\end{proof}

The condition $\beta_1^{(2)}(\pi)=0$ holds if $\pi$ has an infinite normal subgroup which is amenable,
or is finitely generated and of infinite index.
(See Chapter 7 of \cite{Lu}.)

\begin{lemma}
Let $G$ be a group.
If $T$ is a locally finite normal subgroup of $G$ then $T$ 
acts trivially on $H^j(G;\mathbb{Z}[G])$, for all $j\geq0$.
\end{lemma}

\begin{proof} 
If $T$ is finite then  $H^j(G;\mathbb{Z}[G])\cong{H^j}(G/T;\mathbb{Z}[G/T]))$,
for all $j$, and the result is clear.
Thus we may assume that $T$ and $G$ are infinite.
Hence $H^0(G;\mathbb{Z}[G])=0$, and $T$ acts trivially.
We may write $T=\cup_{n\geq1}{T_n}$ as a strictly increasing union of finite subgroups. 
Then there are short exact sequences \cite{Je}
\[
0\to{\varprojlim}^1{H^{s-1}(T_n;\mathbb{Z}[\pi])}\to{H^s(T;\mathbb{Z}[\pi])}\to
\varprojlim{H^s(T_n;\mathbb{Z}[\pi])}\to0.
\]
Hence $H^s(T;\mathbb{Z}[\pi])=0$ if $s\not=1$ and 
$H^1(T;\mathbb{Z}[\pi])={\varprojlim}^1H^0(T_n;\mathbb{Z}[\pi])$,
and so the LHS spectral sequence collapses to give
$H^j(G;\mathbb{Z}[G])\cong{H^{j-1}}(G/T;H^1(T;\mathbb{Z}[G]))$,
for all $j\geq1$.
Let $g\in{T}$. 
We may assume that $g\in{T_n}$ for all $n$, and so $g$ acts trivially on $H^0(T_n;\mathbb{Z}{G})$,
for all $j$ and $n$.
But then $g$ acts trivially on ${\varprojlim}^1H^0(T_n;\mathbb{Z}[\pi])$,
by the functoriality of the construction.
Hence every element of $T$ acts trivially on $H^{j-1}(G/T;H^1(T;\mathbb{Z}[G]))$,
for all $j\geq1$.
\end{proof}

If $\pi$ is a group and $M$ is a right $\mathbb{Z}[\pi]$-module let $\overline{M}$ 
be the conjugate left module, with $\pi$-action defined by $g.m=mg^{-1}$ for all $m\in{M}$ and
$g\in\pi$.

\begin{theorem}
Let $X$ be an orientable $PD_4$-complex with fundamental group $\pi$.
If $X$ is strongly minimal and $\pi$ has one end then 
$\pi$ has no non-trivial locally-finite normal subgroup.
\end{theorem}

\begin{proof}
Let $\widetilde{X}$ be the universal cover of $X$,
and let $C_*=C_*(X;\mathbb{Z}[\pi])=C_*(\widetilde{X};\mathbb{Z})$
be the cellular chain complex of $\widetilde{X}$.
We assume that $\pi$ acts on the left on $\widetilde{X}$, 
and so $C_*$ is a complex of free left $\mathbb{Z}[\pi]$-modules.
Since $\pi$ has one end,  $H_s(\widetilde{X};\mathbb{Z})=H_s(X;\mathbb{Z}[\pi])=0$ for $s\not=0$ or 2.
Let $\Pi=\pi_2(X)\cong{H_2(\widetilde{X};\mathbb{Z})}=H_2(X;\mathbb{Z}[\pi])$.
Poincar\'e duality and $c_X$ give an isomorphism $\Pi\cong\overline{H^2(X;\mathbb{Z}[\pi])}$.
Since $X$ is strongly minimal, this in turn is isomorphic to $\overline{H^2(\pi;\mathbb{Z}[\pi])}$.

Suppose that $\pi$ has a nontrivial locally-finite normal subgroup $T$.
Let $g\in{T}$ have prime order $p$, and let $C=\langle{g}\rangle\cong{Z/pZ}$.
Then $C$ acts freely on $\widetilde{X}$,
which has homology only in degrees 0 and 2.
On considering the homology spectral sequence for the classifying map
$c_{\widetilde{X}/C}:\widetilde{X}/C\to{K(C,1)}$, 
we see that $H_{i+3}(C;\mathbb{Z})\cong H_i(C;\Pi)$, for all $i\geq2$.
(See Lemma 2.10 of \cite{Hi}.)
Since $C$ has cohomological period 2 and acts trivially on $\Pi$, 
by Lemma 1,
there is an exact sequence 
\[
0\to{Z/pZ}\to\Pi\to\Pi\to0.
\]
On the other hand, since $\pi$ is finitely presentable,
$\Pi\cong{H^2(\pi;\mathbb{Z}[\pi])}$ is torsion-free,
by Proposition 13.7.1 of \cite{Ge}.
Hence $T$ has no such element $g$ and so $\pi$ has no such finite normal subgroup.
\end{proof}
               
\begin{cor}
Every ascending locally finite subgroup of $\pi$ is trivial.
\end{cor}

\begin{proof}
If $T$ is an ascending  locally finite subgroup of $\pi$ then 
a transfinite induction shows that the normal closure of $T$ in $\pi$ is locally finite.
\end{proof}

Theorems 8.1 and 9.1 of \cite{Hi} are the bases for characterizations 
of the homotopy types of 4-dimensional infrasolvmanifolds and Seifert fibred 4-manifolds,
respectively.
Each assumes at some point that $H^2(\pi;\mathbb{Z}[\pi])=0$.
We have long felt that this condition might be redundant.
The results above justify this expectation for Theorem 9.1,
which may now be stated as follows.
 
\begin{thm} 
Let $M$ be a closed $4$-manifold with fundamental group $\pi$.
If $\chi(M)=0$ and $\pi$ has an elementary amenable normal subgroup $\rho$ with
$h(\rho)=2$ and $[\pi:\rho]=\infty$ then $M$ is aspherical and $\rho$ is virtually abelian. 
\qed
\end{thm}
 
The key points are that $\pi$ has one end, since $h(\rho)>1$,
and  the quotient of $\rho$ by its maximal locally-finite 
normal subgroup $\tau$ is virtually solvable,
since $h(\rho)<\infty$.
Hence $\tau=1$, so $\rho$ is virtually solvable, 
and then $\pi$ has a nontrivial torsion-free abelian normal subgroup.
Theorem 1.17 of \cite{Hi} may then be used to show 
that $H^2(\pi;\mathbb{Z}[\pi])=0$,
and the rest of the argument is as in \cite{Hi}.
 
We do not yet have a corresponding improvement of Theorem 8.1,
as we do not know how to show that  $H^2(\rho;W)=0$ 
if $\rho$ is elementary amenable, 
has no nontrivial locally-finite normal subgroup and $h(\rho)=\infty$,
and $W$ is a free $\mathbb{Z}[\rho]$-module.
 (This holds if $\rho$ is torsion-free \cite{Kr}.)

\section{centres, hirsch-plotkin radicals and virtually solvable 2-knot groups}

We work with topological manifolds and locally flat embeddings.
Thus a 2-knot is an embedding $K:S^2\to{S^4}$ which extends to 
a product neighbourhood $N(K)\cong{S^2\times{D^2}}$.
The knot exterior is $X(K)=\overline{S^4\setminus{N(K)}}$,
and the knot group is $\pi=\pi{K}=\pi_1(X(K))$.
Let $M(K)=X(K)\cup{D^3\times{S^1}}$ be the closed 4-manifold 
obtained by elementary surgery on $K$.
Then $\pi_1(M(K))\cong\pi$ and $\chi(M(K))=0$.

We can recover the exterior from the knot manifold
$M(K)$ by specifying the isotopy class of the core $\{0\}\times{S^1}$.
By general position, this is determined by the conjugacy class of a meridian.
A {\it weight orbit\/} in a knot group $\pi$ is the orbit 
of a normal generator under the action of $Aut(\pi)$.
If each automorphism of $\pi$ is realizable by a self-homeomorphism 
of $M(K)$ then $X(K)$ is determined by $M(K)$ together with 
the weight orbit of a meridian for $K$.
(In general, a 2-knot group may have many distinct weight orbits,
corresponding to distinct knots giving rise to the same knot manifold.
See \S7 of Chapter 14 of \cite{Hi}.)

If we reattach ${S^2\times{D^2}}$ to $X(K)$
via an identification $\partial{X(K)}\cong{S^2\times{S^1}}$
we may recover a 2-knot, represented by the core $S^2\times\{0\}$. 
However there is a potential ambiguity of order 2 here, 
since $S^n\times{S^1}$ has an involution which does not extend 
across $S^n\times{D^2}$, when $n>1$ \cite{Gl}. 
If this involution extends across $X(K)$ then $K$ is determined by its exterior;
we say that $K$ is {\it reflexive}.

\begin{theorem}
Let $K$ be a $2$-knot with group $\pi=\pi{K}$.
If $\pi$ has normal subgroups $A\leq{E}$ with $A$ a nontrivial abelian group
and $E$ an infinite elementary amenable group 
then either $\pi'$ is finite or $E$ is virtually torsion-free solvable.
If $h(E)=1$ then $E$ is abelian or virtually $\mathbb{Z}$;
if $h(E)=2$ then $E\cong\mathbb{Z}^2$,
and if $h(E)>2$ then $E$ is torsion-free polycyclic, and $h(E)=3$ or $4$.
\end{theorem}

\begin{proof}
We may assume that $\pi'$ is infinite.
Then $\pi$ has one end,
and $\beta_1^{(2)}(\pi)=0$, 
since $\pi$ has an infinite amenable normal subgroup.
Since $M(K)$ is a closed 4-manifold,
it is homotopy equivalent to a finite $PD_4$-complex, 
and so is strongly minimal, by Lemma 1.
The torsion subgroup of $A$ is characteristic, 
and so is normal in $\pi$.
Hence $A$ is torsion-free, by Theorem 3.
Therefore either $\pi\cong\Phi$ or $M(K)$ is aspherical
or $A\cong\mathbb{Z}$ and $\pi/A$ has infinitely many ends,
by Theorem 15.7 of \cite{Hi}.
If $M(K)$ is aspherical then $E$ has finite cohomological dimension, 
and so is virtually solvable \cite{HL}. 
(See Corollary 1.9.2 of \cite{Hi}.)
If $\pi/A$ has infinitely many ends then $E/A$ is finite, 
since it is an elementary amenable normal subgroup in $\pi/A$.
In all cases $E$ is in fact virtually torsion-free solvable.

Since $\pi$ has one end, $E$ has no non-trivial finite normal subgroup.
Hence if $h(E)=1$ then $\sqrt{E}$ is torsion-free, abelian 
and of index $\leq2$ in $E$.
If $E$ is not finitely generated, neither is $\sqrt{E}$. 
In this case $\pi\cong\Phi$ or $M(K)$ is aspherical
(by Theorem 15.7 of \cite{Hi}, as above), 
so $\pi$ is torsion-free, and $E$ must be abelian.

If $h(E)=2$ then $M(K)$ is aspherical, 
$E$ is torsion-free and $\pi/E$ is virtually a $PD_2$-group,
by Theorems 9.1 and 16.2 of \cite{Hi}.
Since $M(K)$ is orientable, 
$E$ cannot be the Klein bottle group,
and so $E\cong\mathbb{Z}^2$.
If $h(E)>2$ then $\pi$ is torsion-free polycyclic and $h(\pi)= 4$, 
by Theorem 8.1 of \cite{Hi},
so $h(E)=3$ or 4.
\end{proof}

\begin{cor}
The following are equivalent:
\begin{enumerate}
\item$\pi$ is elementary amenable and $h(\pi)<\infty$;
\item $\pi$ is virtually solvable;
\item{either} $\pi'$ is finite or $\pi\cong\Phi$
or $\pi$ is torsion-free polycyclic and $h(\pi)=4$.
\end{enumerate}
\end{cor}

\begin{proof}
If $\pi$ is elementary amenable and $h(\pi)<\infty$ then either $\pi'$ 
is finite or $\pi$ has one end.
In the latter case $\pi$ has no nontrivial locally-finite normal subgroup, 
by Lemma 1 and Theorem 3,
and so $\pi$ is virtually solvable, by Corollary 1.9.2 of \cite{Hi}.

If $\pi$ has a solvable normal subgroup $S$ of finite index
then the lowest nontrivial term of the derived series for $S$ is abelian,
and is characteristic in $S$.
Hence it is normal in $\pi$, and Theorem 4 applies.

Finally, (3) clearly implies the other conditions.
\end{proof}

It is enough to assume that $\pi$ is elementary amenable 
and has a nontrivial abelian normal subgroup.
Can we relax ``virtually solvable" further to ``elementary amenable",
or even just ``amenable"?

If $\pi\cong\Phi$ then $K$ is TOP isotopic to Fox's Example 10 or its reflection Example 11 \cite{Hi09}, 
while if $\pi$ is  torsion-free polycyclic then it is one of the groups described in \S4 of Chapter 16 of \cite{Hi},
and $\pi$ determines $M(K)$ up to homeomorphism,
by Theorem 17.4 of \cite{Hi}.

\begin{cor}
The Hirsch-Plotkin radical $\sqrt\pi$ is nilpotent.
\end{cor}

\begin{proof}
Since $\sqrt\pi$ is locally nilpotent it has a maximal locally-finite normal subgroup 
$T$ with torsion-free quotient.
If $\sqrt\pi$ is finitely generated there is nothing to prove.
Otherwise, $\pi$ has one end and so $T$ is trivial, by Theorem 3.
If  $T=1$ and $h(\sqrt\pi)\leq2$ then $\sqrt\pi$ is abelian;
if $h(\sqrt\pi)>2$ then $\pi$ is virtually polycyclic and $h(\pi)=4$,
by Theorem 8.1 of \cite{Hi}.
In each case $\sqrt\pi$ is nilpotent.
\end{proof}

Every finitely generated abelian group 
is the centre of some high-dimensional knot group \cite{HK78}.
On the other hand, the only 1-knots whose groups have 
nontrivial abelian normal subgroups are the torus knots,
for which $\sqrt\pi=\zeta\pi\cong\mathbb{Z}$ and $\zeta\pi\cap\pi'=1$.
The intermediate case of 2-knots is less clear.
If $\zeta\pi$ has rank $>1$ then it is $\mathbb{Z}^2$; 
most twist spins of torus knots have such groups.
There are examples with centre 1, $Z/2Z$, $\mathbb{Z}$ or $\mathbb{Z}\oplus{Z/2Z}$.
(See Chapters 15--17 of \cite{Hi}.)

\begin{cor}
Let $K$ be a $2$-knot with group $\pi=\pi{K}$.
Then
\begin{enumerate}
\item{if} $\pi$ has two ends then $\zeta\pi\cong\mathbb{Z}\oplus{Z/2Z}$ 
if $\pi'$ has even order;
otherwise $\zeta\pi\cong\mathbb{Z}$;
\item{if} $\pi$ has one end then $\zeta\pi\cong\mathbb{Z}^2$,
or is torsion-free of rank $\leq1$;
\item{if} $\pi$ has infinitely many ends then $\zeta\pi$ is finite.
\end{enumerate}
\end{cor}

\begin{proof}
If $\pi$ has two ends then $\pi'$ is finite, and so $\zeta\pi$ is finitely generated and of rank 1.
It follows from the classification of such 2-knot groups 
(see Chapter 15 of \cite{Hi}) that 
$\zeta\pi\cong\mathbb{Z}\oplus{Z/2Z}$ if $\pi'$ has even order, and
otherwise $\zeta\pi\cong\mathbb{Z}$.

Part (2) follows from Theorem 4, while part (3) is clear.
\end{proof}

Note that $\pi$ has finitely many ends if $\pi'$ is finitely generated, 
or if $\zeta\pi$ is infinite.
When $\pi$ has more than one end Lemma 2.10 of \cite{Hi} 
either does not lead to a contradiction or does not apply.

If $\zeta\pi$ is a nontrivial torsion group then it is finite.
Yoshikawa constructed an example of a 2-knot whose group $\pi$ has centre of order 2 \cite{Yo}.
It is easy to see that $\sqrt\pi=\zeta\pi$ in this case.
The construction may be extended as follows. 
Let $q>0$ be odd and let $k_q$ be a 2-bridge knot such that the 2-fold branched cyclic cover of $S^3$,
branched over $k_q$ is a lens space $L(3q,r)$, for some $r$ relatively prime to $q$.
Let $K_1=\tau_2k_q$ be the 2-twist spin of $k_q$, and let $K_2=\tau_3k$ 
be the 3-twist spin of a nontrivial knot $k$. 
Let $\gamma$ be a simple closed curve 
in $X(K_1)$ with image $[\gamma]\in\pi{K_1}$ of order $3q$, 
and let $w$ be a meridian for $K_2$.
Then $w^3$ is central in $\pi{K_2}$.
The group of the satellite of $K_1$ about $K_2$ relative to $\gamma$ is the generalized free
product
\[
\pi=\pi{K_2}/\langle\langle{w^{3q}}\rangle\rangle*_{w=[\gamma]}\pi{K_1}.
\]
(See \cite{Ka}.)
Hence $\sqrt\pi=\langle{w^3}\rangle\cong{Z/qZ}$, while $\zeta\pi=1$.

If we use a 2-knot $K_1$ with group $(Q(8)\times{Z/3qZ})\rtimes_\theta\mathbb{Z}$
instead and choose $\gamma$ so that $[\gamma]$ has order $6q$ then we obtain
examples with $\sqrt\pi\cong{Z/2qZ}$ and $\zeta\pi=Z/2Z$.
(Knots $K_1$ with such groups may be constructed by surgery on sections of mapping 
tori of homeomorphisms of 3-manifolds with fundamental group $Q(8)\times{Z/3qZ}$ \cite{Yo80}.)

If $\zeta\pi$ has rank 1 and nontrivial torsion then $\pi'$ is finite,
and $\zeta\pi$ is finitely generated.

If $\zeta\pi$ has rank 1 but is not finitely generated then $M(K)$ is aspherical.
It is not known whether there are such 2-knots 
(nor, more generally,
whether abelian normal subgroups of rank 1 in $PD_n$-groups 
with $n>3$ must be finitely generated).
What little we know about this case is as follows.
Since $\zeta\pi<\pi'$ and $\pi/\pi'\cong\mathbb{Z}$, we must have $\zeta\pi\leq\pi''$.
Since $\zeta\pi$ is torsion-free of rank 1 but is not finitely generated, $c.d.\zeta\pi=2$.
Hence if $G$ is a nonabelian subgroup which contains $\zeta\pi$ then $c.d.G\geq3$,
by Theorem 8.6 of \cite{B}.
If $H$ is a subgroup of $\pi$ such that $H\cap\zeta\pi=1$ then $H.\zeta\pi\cong{H}\times\zeta\pi$
is not finitely generated, and so has infinite index in $\pi$.
Hence $c.d.H\times\mathbb{Z}\leq{c.d.}{H}\times\zeta\pi\leq3$ \cite{Str}.
Theorem 5.5 of \cite{B} gives, firstly, that $c.d.H\leq2$, 
and then, that if $H$ is $FP_2$ then $c.d.H\leq1$, and so $H$ is free.
Thus if $\pi$ is almost coherent every subgroup either meets $\zeta\pi$ nontrivially
or is locally free.

If $\zeta\pi$ has rank $>1$ then $M(K)$ is aspherical and
$\zeta\pi\cong\mathbb{Z}^2$, by Theorem 16.3 of \cite{Hi}.

The following questions remain open:
 \begin{enumerate}[(i)]
 \item{if} $\zeta\pi$ has rank 1, must it be finitely generated?
 
 \item{if} $\zeta\pi$ is finite, must it be $Z/2Z$ or 1?
 
 \item{is} there a 2-knot group $\pi$ with $\sqrt\pi$ a non-cyclic finite group?
 
\item{if} $\pi$ is elementary amenable is it virtually solvable?
\end{enumerate}
In each case the answer is ``yes" if $\pi'$ is finitely presentable,
for then the infinite cyclic cover $M(K)'$ is homotopy equivalent 
to a $PD_3$-complex,
by Theorem 4.5 of \cite{Hi}.
  
\section{geometric dimension 2}

The only solvable 2-knot groups with deficiency 1 are $\mathbb{Z}$ and $\Phi$.
One of the early triumphs of 4-dimensional surgery was Freedman's 
unknotting criterion:
2-knots with group $\mathbb{Z}$ are topologically trivial.
In \cite{Hi09} we showed that any 2-knot $K$ with group $\Phi$ 
is topologically equivalent to Fox's Example 10 or its reflection (Example 11).
The argument involved showing that 
\begin{enumerate}[(i)]
\item $\pi=\pi{K}$ determines the homotopy type of  $M(K)$;
\item $Wh(\pi)=0$ and the surgery assembly homomorphisms $\mathcal{A}_*(\pi):H_*(\pi;\mathbb{L}_0)\to{L_*(\pi)}$ are isomorphisms;
\item4-dimensional TOP surgery;
\item $\pi$ has an unique  weight orbit, up to inversion; and 
\item{Fox's} examples are reflexive.
\end{enumerate}
The group $\Phi$ has geometric dimension 2 ($g.d.\Phi=2$): there is a finite 2-dimensional $K(\Phi,1)$ complex.
Our strategy for (i) uses the fact that $M(K)$ is strongly minimal
if $g.d.\pi=2$, since $\chi(M(K))=0=2\chi(\pi)$.
It should extend to all knot groups $\pi$ with $g.d.\pi=2$, 
but at present is incomplete; 
there is a 2-torsion condition which holds for $\Phi$,
but is not otherwise easily verified. 
(See \cite{Hi13} for a much expanded exposition of \cite{Hi09} 
and earlier work.)
When $g.d.\pi=2$ and $\pi'$ is finitely generated 
there is a quite different argument.
For then $\pi'$ is free of finite rank, 
and $M(K)$ is homotopy equivalent to the mapping torus 
of a self-homeomorphism of $\#^r(S^2\times{S^1})$,
where $r$ is the rank of $\pi'$.
(See Corollary 4.5.3 of \cite{Hi}.)

Step (ii) holds if $\pi$ is a 1-knot group, 
or is square root closed accessible,
or is in the class $\mathcal{X}$ of fundamental groups of finite graphs 
of groups with all vertex and edge groups $\mathbb{Z}$.
(See Lemma 6.9 of \cite{Hi}.)
Semidirect products $F(r)\rtimes\mathbb{Z}$ are square root closed accessible, 
while $\Phi$ is in $\mathcal{X}$.

At the time of writing,
the largest known class of groups for which 4-dimensional topological surgery works 
is the class $SG$ obtained from subexponential groups 
by taking increasing unions and extensions.
To go beyond this class we must work modulo TOP $s$-cobordisms 
(or $s$-concordances of 2-knots) rather than expect homeomorphisms.

Step (iv) fails for other knot groups $\pi$ with $g.d.\pi=2$.
In general, we must specify a weight orbit in order to
recover $X(K)$ from $M(K)$.

The final step (v) is not satisfied by all 2-knots.
However, those which have a Seifert hypersurface with free fundamental group 
are reflexive.
Fibred 2-knots with $\pi'$ free and ribbon 2-knots have 
such Seifert hypersurfaces \cite{Ya}.

The difficulty with (i) has been bypassed in \cite{HKT}, 
which shows directly that an orientable 4-manifold $M$ with $\pi=\pi_1(M)$ 
such that $g.d.\pi=2$ and $\pi$ ``satisfies properties $W$-$AA$" 
(i.e., (ii) holds) is determined up to $s$-cobordism by $\pi$, 
the second Wu class $w_2(M)$ and 
the equivariant intersection pairing on $\pi_2(M)$.

\begin{theorem}
Let $\pi$ be a $2$-knot group with deficiency $1$.
If either
\begin{enumerate}
\item$\pi$ is the group of a $1$-knot; or
\item$\pi'$ is finitely generated; or
\item$\beta_1^{(2)}(\pi)=0$ and $\pi$ is square-root closed accessible,
\end{enumerate}
then any $2$-knot with group $\pi$ is determined up to TOP $s$-concordance, 
reflections and the Gluck ambiguity by $\pi$ 
and the weight orbit of a meridian.
\end{theorem}

\begin{proof}
We may assume that $\pi\not\cong\mathbb{Z}$.

If $\pi$ is a 1-knot group then $g.d.\pi=2$ and $\beta_1^{(2)}(\pi)=0$,
by Theorem 4.1 of \cite{Lu},
while  $Wh(\pi)=0$ and the surgery assembly homomorphisms  $\mathcal{A}_*(\pi)$
are isomorphisms \cite{AFR}.
If $\pi'$ is finitely generated then $\beta_1^{(2)}(\pi)=0$,
while if $\mathrm{def}(\pi)=1$ and $\beta_1^{(2)}(\pi)=0$ then
$g.d.\pi=2$, by Theorem 2.5 of \cite{Hi}.
In particular, if $\pi'$ is finitely generated and  $\mathrm{def}(\pi)=1$ 
then $\pi'$  is free and so $\pi\cong\pi'\rtimes\mathbb{Z}$ is square root closed accessible.
Hence $(2)\Rightarrow(3)$, which implies that $Wh(\pi)=0$ and the  $\mathcal{A}_*(\pi)$
are isomorphisms \cite{Ca}.
Thus in each case  $\pi$ satisfies properties $W$-$AA$ of \cite{HKT}.

In each case $\pi$ is torsion-free, 
and so has one end \cite{Kl}, and $\beta_1^{(2)}(\pi)=0$. 
Hence $M(K)$ is strongly minimal, by Lemma 1.
Since $M(K)$ is orientable, 
$w_2(M(K))=0$, the equivariant intersection pairing on $\pi_2(M(K))$
is trivial and $\pi$ satisfies properties $W$-$AA$, 
$\pi$ determines $M(K)$ up to $s$-cobordism,
by the main result of \cite{HKT}.
More precisely, if $K_0$ and $K_1$ each have group $\pi$ and 
their weight orbits correspond under an isomorphism $\theta:\pi{K_0}\cong\pi{K_1}$ 
then there is a (5-dimensional) $s$-cobordism $W$ from $M(K_0)$ to $M(K_1)$
such that $\pi_1(j_0)=\pi_1(j_1)\theta$, where $j_i:M(K_i)\to{W}$ 
is the natural inclusion, for $i=0,1$.
The weight orbit of a meridian determines a properly embedded annulus
in $W$ with boundary components meridians for $K_0$ and $K_1$.
Excising an open product neighbourhood of this annulus gives 
a relative $s$-cobordism between the exteriors, 
and hence an $s$-concordance between the knots, 
modulo reflections and the Gluck ambiguity.
\end{proof}

Beyond this, there is some reason to hope that 
all 2-knots with groups of deficiency 1 should be so determined.
The exact state of affairs is contingent upon the outcome 
of several well-known conjectures.
Let $\mathcal{DW}$ be the class of groups with deficiency 1 and weight 1.
Kervaire showed that every group in $\mathcal{DW}$  is the group 
of a smooth knot in a smooth homotopy 4-sphere \cite{Ke}.
He first constructs a knot manifold as the boundary of a 5-manifold obtained from $D^5$
by attaching 1- and 2-handles, and so the groups are realized by homotopy ribbon knots.

If the {\it Andrews-Curtis Conjecture\/} is true then groups 
in $\mathcal{DW}$ have Wirtinger presentations of deficiency 1, 
and so are groups of ribbon 2-knots \cite{Yo82'}.
(In particular, these are smooth knots in $S^4$.) 
If the {\it Whitehead Conjecture\/} is true then $g.d.\pi\leq2$, 
for all $\pi\in\mathcal{DW}$.
(If $\pi\in\mathcal{DW}$ then 
$g.d.G=2\Leftrightarrow{c.d.G=2}$, by Theorem 2.8 of \cite{Hi}.)
If the {\it Fibred Isomorphism Conjecture\/} holds for groups 
of cohomological dimension 2 then (ii) holds for such groups.

A conjecture which has perhaps been less studied has implications for 
characterizing 1-knot groups, 
and thus for recognizing the Artin spins of 1-knots.
If $\pi$ is a nontrivial 1-knot group and $T\cong\mathbb{Z}^2$ 
is the subgroup generated by a meridian-longitude pair 
$(\mu, \lambda)$ then $(\pi,T)$ is a $PD_3$-pair.
(When the knot is prime the subgroup $T$ is unique up to conjugacy.)
If {\it finitely presentable $PD_3$-groups are fundamental groups 
of aspherical closed 3-manifolds\/} then the converse holds:
every $PD_3$-pair $(\pi,T)$
in which $T\cong\mathbb{Z}^2$ contains a normal generator 
is the group of a 1-knot.

This is known in the fibred case.
The following result is essentially Theorem 5 of \cite {Th},
which is in turn a reformulation of Theorem 9.2.3 of \cite{Ne},
taking into account first
the determination of $PD_2$-groups and pairs, 
and more recently
Perelman's proof of the Poincar\'e Conjecture. 

\begin{theorem} 
Let $\pi$ be a group with deficiency $1$ and weight $1$, 
and suppose that $\pi\not\cong\mathbb{Z}$.
Then the following are equivalent:
\begin{enumerate}
\item $\pi$ is the group of a fibred $1$-knot ;
\item$\pi$ has a subgroup $\Delta\cong\mathbb{Z}^2$ which contains a normal generator 
and such that $(\pi,\Delta)$ is a $PD_3$-pair, and $\pi'$ is finitely generated;
\item $\pi$ has a subgroup
$\Delta=\langle\mu,\lambda\rangle\cong\mathbb{Z}^2$, where $\mu$
is a normal generator and $\lambda\in\pi''$, and such that
$(\pi',\langle\lambda\rangle)$ is a $PD_2$-pair.
\end{enumerate}
\end{theorem}

\begin{proof}
If $\pi\cong\pi{k}$, where $k$ is a nontrivial 1-knot,
then $X(k)$ is aspherical, $\partial{X(k)}$ is an incompressible torus,
and $\pi$ is normally generated by a meridian on $\partial{X(k)}$.
Hence $\Delta=\pi_1(\partial{X(k)})\cong\mathbb{Z}^2$ contains a normal generator
and $(\pi,\Delta)$  is a $PD_3$-pair. 
If $X(k)$ fibres over $S^1$ with fibre $F$ then $\pi'\cong\pi_1(F)$ is finitely generated.
Thus  (1) implies (2).

If (2) holds then $\pi/\pi'\cong\mathbb{Z}$, and so $T$ has a basis $\mu,\lambda$, where 
$\mu$ is a normal generator and $\lambda$ generates $T\cap\pi'\cong\mathbb{Z}$.
In fact, $\lambda$ is in $\pi''$, since $\mu$ commutes with $\lambda$ and $\pi/\pi'\cong\mathbb{Z}$.
Since $\mathrm{def}(\pi)=1$ and $\pi'$ is finitely generated,
$\pi'$ is free, and so is $FP$.
Hence $(\pi',\langle\lambda\rangle)$ is a $PD_2$-pair,
by a Shapiro's Lemma argument, as in Theorem 1.19 of \cite{Hi}.
Thus (2) implies (3).

If (3) holds then $(\pi',\langle\lambda\rangle)$ is 
the fundamental group pair of a surface $S$ with boundary $\partial{S}=S^1$
\cite{EM}.
The group $\pi$ is a semidirect product  $\pi\cong\pi'\rtimes_\theta\mathbb{Z}$,
where the monodromy $\theta$ fixes $\lambda$.
Hence it is realized by a self-homeomorphism $f$ of $S$,
by Nielsen's Theorem.
(See Theorem 5.7.1 of \cite{ZVC}.)
We may adjoin a solid torus to the mapping torus $M(f)$
to get a homotopy 3-sphere containing a knot with exterior $M(f)$.
Thus $\pi$ is the group of a fibred 1-knot, 
by the validity of the Poincar\'e Conjecture.
\end{proof}

\begin{cor}
A nontrivial $2$-knot with group $\pi$ and meridian $\mu$  
is $s$-concordant to the Artin spin of a fibred $1$-knot 
if and only if $\pi'$ is finitely generated, 
$C_\pi(\mu)\cong\mathbb{Z}^2$ 
and $(\pi,C_\pi(\mu))$ is a $PD_3$-pair.
\end{cor}

\begin{proof}
This follows immediately from Theorems 5 and 6, 
since Artin spins of 1-knots are reflexive
\cite{Gl}.
\end{proof}



\begin{thebibliography}{99}

\bibitem{AFR} Aravinda, C.S., Farrell, F.T. and Roushon, S.K. 
Surgery groups of knot and link complements, 
Bull. London Math. Soc. 29 (1997), 400--406.

\bibitem{B} Bieri, R. {\it Homological Dimension of Discrete Groups},
Queen Mary College Lecture Notes in Mathematics, London (1976).

\bibitem{Ca} Cappell, S.E. Mayer-Vietoris sequences in Hermitean 
$K$-theory, 
in \textit{ Hermitean $K$-Theory and Geometric Applications} 
(edited by H.Bass), 
Lecture Notes in Mathematics 343, 
Springer-Verlag, Berlin - Heidelberg - New York (1973), 478--512.

\bibitem{EM} Eckmann, B. and M\"uller, H. Poincar\'e duality groups 
of dimension two, 
Comment. Math. Helvetici 55 (1980, 510--520.

\bibitem{Fo} Fox, R.H. A quick trip through knot theory,
in {\it Topology of 3-Manifolds and Related Topics} 
(edited by M.K.Fort, Jr), 
Prentice-Hall, Englewood Cliffs, N.J.(1962), 120--167.

\bibitem{Ge} Geoghegan, R.  {\it Topological Methods in Group Theory},
Graduate Texts in Mathematics 243,
Springer-Verlag, Berlin -- Heidelberg -- New York (2008).

\bibitem{Gl} Gluck, H. The embedding of two-spheres in the four-sphere,
Trans. Amer. Math. Soc. 104 (1962), 308--333.

\bibitem{Go} Gompf, R. More Cappell-Shaneson spheres are standard,

Alg. Geom. Top 10 (2010),   1665--1681.

\bibitem{HKT} Hambleton, I., Kreck, M. and Teichner, P. Topological 
4-manifolds with geometrically two-dimensional fundamental groups,
J. Topol. Anal. 1(2009), 123--151.

\bibitem {HK78} Hausmann, J.-C. and Kervaire, M. 
Sur le centre des groupes de noeuds multidimensionelles, 
C.R. Acad. Sci. Paris 287 (1978), 699--702.

\bibitem{Hi} Hillman, J. A. {\sl Four-Manifolds, Geometries and Knots},
Geometry and Topology Monographs 5, 
Geometry and Topology Publications (2002). (Revisions 2007 and 2014).

\bibitem{Hi09} Hillman, J.A. Strongly minimal $PD_4$-complexes,
Top. Appl. 156 (2009), 1565--1577.

\bibitem{Hi13} Hillman, J.A. $PD_4$-complexes and 2-dimensional duality groups,
arXiv: 1303.5486 v7 [math.GT]

\bibitem{HH} Hillman, J.A. and Howie, J.  Seifert fibred knot manifolds,
J. Knot Theory Ramif. 22 (2013), 1350082. 

\bibitem {HL} Hillman, J.A. and Linnell, P.A.  Elementary amenable groups 
of finite Hirsch length are locally finite by virtually solvable,
J. Austral. Math. Soc. 52 (1992), 237--241.

\bibitem{Je} Jensen, C.U. {\sl Les foncteurs d\'eriv\'ees de $\varprojlim$ 
et ses applications a la th\'eorie des modules},
Lecture Notes in Mathematics 254,
Springer-Verlag, Berlin -- Heidelberg -- New York (1972).

\bibitem{Ka} Kanenobu,T. Groups of higher dimensional satellite knots,

J. Pure Appl. Alg. 28 (1983), 179--188.

\bibitem{Ke} Kervaire, M.A. Les noeuds de dimensions sup\'erieures,

Bull. Soc. Math. France 93 (1965), 225--271.

\bibitem{Kl} Klyachko, A. Funny property of sphere and equations over groups,
Comm. Alg. 21 (1993), 2555--2575.

\bibitem{Kr} Kropholler, P.H. Soluble groups of type $(FP)_\infty$ 
have finite torsion-free rank,
Bull. London Math. Soc. 25 (1993), 558--566.

\bibitem {Lu} L\"uck, W. 
{\it $L^2$-Invariants: Theory and Applications to Geometry and $K$-Theory},
Ergebnisse der Mathematik und ihrer Grenzgebiete 3 Folge, Bd. 44,
Springer-Verlag, Berlin - Heidelberg - New York (2002).

\bibitem{Ne} Neuwirth, L.P. {\it Knot Groups},
Ann. Math. Studies 56,   Princeton University Press, Princeton, N.J. (1965).

\bibitem{Ro} Robinson, D.J.S. {\sl A Course in the Theory of Groups},
Graduate Texts in Mathematics 80,
Springer-Verlag, Berlin -- Heidelberg -- New York (1982).

\bibitem{Str} Strebel, R. A remark on subgroups of infinite index 
in Poincar\'e duality groups,
Comment. Math. Helvetici 52 (1977), 317--324.

\bibitem{Th} Thomas, C.B. Splitting theorems for certain $PD^3$-groups,
Math. Z. 186 (1984), 201--209.

\bibitem{Ya} Yanagawa, T. On ribbon 2-knots: the 3-manifold bounded 
by the 2-knot,
Osaka J. Math. 6 (1969), 447-464.

\bibitem{Yo80} Yoshikawa, K.  On 2-knot groups with the finite commutator subgroups,
Math. Seminar Notes Kobe University 8 (1980), 321--330.

\bibitem{Yo} Yoshikawa, K. On a 2-knot group with nontrivial centre,
Bull. Austral.Math. Soc. 25 (1982), 321--326.

\bibitem {Yo82'} Yoshikawa, K. A note on Levine's condition 
for knot groups,
Math. Sem. Notes Kobe University 10 (1982), 633--636.

\bibitem{ZVC} Zieschang, H., Vogt, E. and Coldewey, H.-D. 
{\it Surfaces and Planar Discontinuous Groups},
Lecture Notes in Mathematics 835, 
Springer-Verlag, Berlin -- Heidelberg -- New York (1980).
 
\end{thebibliography}
\end{document}